\documentclass{amsproc}
\usepackage{amsfonts}
\usepackage{amssymb}
\usepackage{amsthm}
\usepackage{amsmath}
\usepackage{graphicx}

\newtheorem{theorem}{Theorem}[section]
\newtheorem{corollary}[theorem]{Corollary}
\newtheorem{lemma}[theorem]{Lemma}
\newtheorem{example}[theorem]{Example}
\newtheorem{definition}[theorem]{Definition}
\newtheorem{proposition}[theorem]{Proposition}

\theoremstyle{definition}

\theoremstyle{remark}
\newtheorem{remark}[theorem]{Remark}

\numberwithin{equation}{section}

\begin{document}

\title{Orbits of linear operators and Banach space geometry}

\author[J. M. Aug\'{e}]{Jean-Matthieu Aug\'{e}}\address{Universit\'{e} Bordeaux 1
351, cours de la Lib\'{e}ration - F 33405 TALENCE cedex }
\email{jean-matthieu.auge@math.u-bordeaux1.fr}


\begin{abstract}
Let $T$ be a bounded linear operator on a (real or complex) Banach space $X$. If $(a_n)$ is a sequence of non-negative numbers tending to 0. Then, the set of $x \in X$ such that $\|T^nx\| \geqslant a_n \|T^n\|$ for infinitely many $n$'s has a complement which is both $\sigma$-porous and Haar-null. We also compute (for some classical Banach space) optimal exponents $q>0$, such that for every non nilpotent operator $T$, there exists $x \in X$ such that $(\|T^nx\|/\|T^n\|) \notin \ell^{q}(\mathbb{N})$, using techniques which involve the modulus of asymptotic uniform smoothness of $X$. 
\end{abstract}



\maketitle

\section{Introduction}

Let $X$ be a (real or complex) Banach space, and let $T$ be a linear bounded operator on $X$. For $x \in X$, let 
$$ O_T(x)=\{T^nx, n \geqslant 0\}$$ be the orbit of $x$ under the action of $T$. The study of orbits is connected with the famous invariant subset problem which asks if there exists an operator on $X$ with non trivial invariant subset. Indeed, $T$ does not have any trivial invariant subset if and only if for each $x \neq 0$, $O_T(x)$ is dense in $X$. Such an operator was constructed by Read \cite{R} in the space $\ell^1$, but in the Hilbert space, the problem is still open. If at least one orbit is dense, the operator is called hyperyclic (and the corresponding vector a hypercyclic vector). This class of operators has received much attention during the last two decades (see \cite{BM} for many informations on this topic). In this paper, we will, however, study some more regular orbits. M\"uller \cite{Mu} showed the following result, which roughly says that there are many points with large orbits for many powers.

\begin{theorem} \label{thm:M} Let $T$ be a bounded linear operator on $X$, and let $(a_n)$ be a sequence of non negative numbers such that $a_n \rightarrow 0$. Then, the set
$$ \{x \in X, \|T^nx\| \geqslant a_n \|T^n\| \; \rm{for \; infinitely \; many \;}  n's\}$$ is residual in $X$.
\end{theorem}

A quick glance at the proof shows that the powers can be replaced by a sequence $(T_n)$ of bounded linear operators. It may also be worth emphasizing that this a stronger form than the uniform boundedness principle. Indeed, suppose that $\sup \|T_n\|=\infty$ and find a sequence $(n_j)$ such that $\|T_{n_j}\| \rightarrow \infty$. Put $a_j=\frac{1}{\sqrt{\|T_{n_j}\|}}$ and apply the above result to $S_j=T_{n_j}$ to obtain that there is a residual set of points $x \in X$ such that $\|T_{n_j}x\| \geqslant \sqrt{\|T_{n_j}\|}$ for infinitely many $j's$, and in particular $\sup \|T_{n}x\|=\infty$. Equivalently, Theorem \ref{thm:M} says that the complement of the set in question is of the first category. Now, there are several other notions of smallness in analysis. In this paper, we consider two of them: $\sigma$-porosity, which is a stronger form of smallness than sets of first category and Haar-negligibility, which is an extension of the sets of Lebesgue-measure $0$ in infinite di
 mension (see next section for definitions). These two notions are actually not comparable: Preiss and Tisier (see \cite{BL}, chapter 6) showed that any real Banach space of infinite dimension can be decomposed as the disjoint union of two sets, of which one is $\sigma$-porous and the other Haar-null. In section 2, we generalize Theorem \ref{thm:M} as follows:

\begin{theorem}
\label{smallness}
Let $X$ be a Banach space (real or complexe) and $(T_n)$ be a sequence of bounded linear operators on $X$. Let also  $(a_n)$ be a sequence of non-negative numbers such that $a_n \rightarrow 0$. Then, the set 
$$ \{x \in X, \|T_nx\| \geqslant a_n \|T_n\| \; \rm{for \; infinitely \; many \;}n's\}$$
has a complement which is $\sigma$-porous. If the space $X$ is separable, then this complement is also Haar-null.
\end{theorem}

These notions of smallness, together with linear dynamics, have also been studied (for different problems) in \cite{Bay1} and \cite{BMM}. The example $T=B$ where $B$ is the unweighted backward shift defined on $\ell^1(\mathbb{N})$ by $Be_1=0$ and $Be_k=e_{k-1}$ for $k \geqslant 2$ (where $(e_k)$ is the canonical basis of $\ell^1$) satisfies for each $x$, $\frac{\|T^nx\|}{\|T^n\|} \rightarrow 0$ because $\|T^n\|=1$ and $\|T^nx\|=\sum_{k=n+1}^{\infty} |x_k|$ for $x=\sum_{k=1}^{\infty} x_k e_k$. Thus, in general, the condition "$a_n \rightarrow 0$" cannot be improved. In section 3, we study some cases involving compact operators to get better estimates. We also give some examples to discuss the limitations of our results. Section 4 contains our main result. As should be clear from what we said, the spirit of this paper is to find point $x \in X$ such that many powers $\|T^nx\|$ are as close as possible from $\|T^n\|$. Beauzamy \cite{Be} showed that given a bounded and linear ope
 rator $T$ on an Hilbert space $H$, the set 
$$ \{x \in H, \sum_{n=1}^{\infty} \dfrac{\|T^nx\|}{\|T^n\|}=\infty\} $$ was dense in $H$ (which in some sense, says that $\|T^nx\|$ is not too far from $\|T^n\|$ for many powers). Using an alternative proof, M\"uller \cite{Mu} showed that for each $q<2$, the set 
$$ \{x \in H, \sum_{n=1}^{\infty} \left( \dfrac{\|T^nx\|}{\|T^n\|} \right)^{q}=\infty\} $$ 
was dense in $H$ and also showed a similar statement for Banach operators (replacing $q<2$ by $q<1$). He also exhibed examples showing that the constants $1$ and $2$ were optimals for Banach space and Hilbert space operators. These problems also have connections with some famous plank theorems of Ball (see \cite{Bal1} and \cite{Bal2}). We recall those results.

\begin{theorem} (K. Ball, \cite{Bal1}) Let $X$ be a (real or complex) Banach space and $(f_n) \subset X^*$ such that $\|f_n\|=1$ for each $n$. Let also $(\alpha_n) \subset \mathbb{R}^+$ such that $\sum_{n=1}^{\infty} \alpha_n<1$. Then there is a point $x$, $\|x\|=1$ such that for each $n$, $ |\langle f_n,x \rangle | \geqslant \alpha_n$.
\end{theorem}

In \cite{Bal2}, the condition $\sum_{n=1}^{\infty} \alpha_n<1$ is improved by $\sum_{n=1}^{\infty} \alpha_n^2<1$ for complex Hilbert spaces. Now, considering the adjoint of $T^n$, one can show that a similar statement holds for sequence of operators (see \cite{MV} for details). This gives a direct proof of the above results. Anyways, the previous exponents suggest that for a Banach space $X$, the quantity
\begin{eqnarray*} q_X=\sup \{q>0, {\rm{for \; every \; non \; nilpotent \; and \; bounded \; linear \; operator}} \; T;
\end{eqnarray*}
$$ \exists x \in X, \sum_{n=1}^{\infty}\left( \dfrac{\|T^nx\|}{\|T^n\|} \right )^{q}=\infty\} $$
should depend on the geometry of $X$. This will be the case and we will in particular obtain:

\begin{theorem} $q_{\ell^p}=p$, $q_{L^p}=\min(p,2)$ ($1 \leqslant p<\infty$), $q_{c_0}=\infty$.
\end{theorem}

To unify those results,  we will use the modulus of asymptotic uniform smoothness of $X$, which is a tool from Banach space geometry that has been used for several problems of Nonlinear Functional Analysis (as the opposite of this paper!). We refer the reader to section 4 for definitions. Note that a similar discussion, involving weakly closed sequences and type of the space can be found in \cite{Bay2} and \cite{BM}, chapter 10. Until the end of the paper, we shall denote by $\mathcal{L}(X)$ the set of linear and bounded operators acting on a Banach space $X$ and by $B(x,r)$ the open ball of center $x$ and radius $r$ ($x \in X, r>0$).

\section{$\sigma$-porosity and Haar-negligibility: proof of Theorem \ref{smallness}}

Let us first recall the definitions of $\sigma$-porous and Haar-null sets. The notion of porosity quantifies the fact that a set has empty interior. Porosity was introduced by P. Dol\v{z}enko \cite{D}, and has been studied in details since then (see \cite{Z}). It appears for example in the study of differentiability properties of real-valued convex functions defined on a separable Banach space $X$.

\begin{definition} A subset $E$ of a Banach space $X$ is called porous if there exists $\lambda \in ]0, 1[$ such that the following is true: for every $x \in E$ and every $\epsilon> 0$, there exists a point $y \in X$ such that $0 < \|y-x\| < \epsilon$ and $E  \cap B(y, \lambda \|x-y\||)$ is empty. A countable union of porous sets is called a $\sigma$-porous set.
\end{definition}

Haar-null sets were introduced by Christensen in \cite{C}. They appear in the study of differentiability of Lipschitz functions defined on a Banach space. Note that the definition of Haar-null sets makes sense in any Polish abelian group $G$. Here, we restrict ourselves to the Banach space setting.

\begin{definition} Let $X$ be a separable Banach space. A set $E \subset X$ is said to be Haar-null if there exists a Borel probability measure $m$ on $X$ such that for every $x \in X$, the translate $x+E$ has $m$-measure $0$.
\end{definition}

With those definitions in mind, we can now start the proof of Theorem \ref{smallness}. Let $(T_n) \subset \mathcal{L}(X)$ and $(a_n) \subset \mathbb{R}^+$ such that $a_n \rightarrow 0$. If infinitely many $T_n$'s are 0, then the result is obvious. We may assume, without loss of generality, that for every $n$, $T_n \neq 0$. Let us prove first the assertion about $\sigma$-porosity.  We can write the complement in the form $\cup_{N=1}^{\infty} E_N$ with 
$$ E_N=\{x, \forall n \geqslant N, \|T_nx\|<a_n \|T_n\|\}. $$
We shall see that for fixed $N \geqslant 1$ , $E_N$ is porous with the constant $\lambda=1/4$ (but, actually, any $\lambda \in ]0,1[$ suits, by adjusting the computations in what follows). Consider now $x \in E_N$ and $\epsilon>0$. Fix successively  $n \geqslant N$ and $y_0$ with $\|y_0\|=1$ such that
$$ a_n \leqslant \frac{\epsilon}{8} \; \; {\rm{and}} \; \; \|T_ny_0\| \geqslant \frac{\|T_n\|}{2}. $$
We have 
$$ \|T_n(x+\frac{\epsilon}{2}y_0)\|+\|T_n(x-\frac{\epsilon}{2}y_0)\| \geqslant \epsilon\|T_ny_0\| \geqslant  \frac{\epsilon}{2} \|T_n\|, $$
so
$$ \|T_n(x+\frac{\epsilon}{2}y_0)\| \geqslant \frac{\epsilon}{4} \|T_n\| \; \; {\rm {or}} \; \; \|T_n(x-\frac{\epsilon}{2}y_0)\| \geqslant \frac{\epsilon}{4} \|T_n\|.
$$ 
Replacing eventually $y_0$ by $-y_0$, we can assume that $\|T_n(x+\frac{\epsilon}{2}y_0)\| \geqslant \frac{\epsilon}{4} \|T_n\|$. Put $y=x+\frac{\epsilon}{2}y_0$, so $\|y-x\|=\frac{\epsilon}{2}<\epsilon$. To conclude, it is enough to show that 
$$ B(y,\frac{1}{4}\|x-y\|) \cap E_N=B(y,\frac{\epsilon}{8}) \cap E_N=\emptyset. $$
Let $z \in B(y,\frac{\epsilon}{8})$, we have 
\begin{eqnarray*}
\|T_nz\| \geqslant \|T_ny\|-\|T_n(z-y)\| & \geqslant & \frac{\epsilon}{4} \|T_n\|-\frac{\epsilon}{8} \|T_n\| \\
& \geqslant & \frac{\epsilon}{8}\|T_n\| \geqslant a_n \|T_n\|
\end{eqnarray*}
and $z \notin E_N$, as announced. \\ 
Let us now prove the second assertion. Considering only the real linear structure of $X$, we may assume that $X$ is a real Banach space. Keeping the same notations as above, it is enough to show that each $E_N$ is Haar-null since a countable union of Haar-null sets is a Haar-null set. Now, we use the following useful criterion \cite{HS}: if $E \subset X$ is a borelian subset of $X$ and if there exists a subspace $V \subset X$ of finite dimension such that
$$ {\rm{for \; all \;}} x  \in X,  {\rm{\; almost \; every}} \; v \in V,\; x+v \notin E, $$
then $E$ is a Haar-null set (here "almost every" refers to the Lebesgue measure on $V$). We can find $u \in X$ such that for infinitely many $n$'s, $\|T_nu\| \geqslant \sqrt{a_n} \|T_n\|$ (because of the first part of the theorem) and we show that $V=\mathbb{R} u$ is the subspace we are looking for. Fix $x \in X$, put  
$$ \Lambda=\{\lambda \in \mathbb{R}, x+\lambda u \in E_N\} $$
and let us check that $\Lambda$ has Lebesgue measure $0$. Let $\lambda \in \Lambda$, then
$$ \|T_n(x+\lambda u) \| \leqslant a_n \|T_n\| \; \; (n \geqslant N).$$
Hence, we get
$$ \left | |\lambda|-\dfrac{\|T_nx\|}{\|T_nu\|} \right | \leqslant a_n \dfrac{\|T_n\|}{\|T_nu\|} \; \; (n \geqslant N). $$
Put $b_n=\|T_nx\|/\|T_nu\|$, the above inequality shows that $\lambda \in E_{+} \cup E_{-}$, where
$$E_+=\bigcap_{n \geqslant N}  [b_n-a_n \dfrac{\|T_n\|}{\|T_nu\|},b_n+a_n \dfrac{\|T_n\|}{\|T_nu\|}] $$
$$E_-= \bigcap_{n \geqslant N} [-b_n-a_n \dfrac{\|T_n\|}{\|T_nu\|},-b_n+a_n \dfrac{\|T_n\|}{\|T_nu\|}]. $$
This implies that the Lebesgue measure of $\Lambda$ is not greater than $4 \inf_{n \geqslant N} a_n \|T_n\|/\|T_nu\|$, which is 0. Indeed for infinitely many $n$'s,
$$ a_n \dfrac{\|T_n\|}{\|T_nu\|} \leqslant \sqrt{a_n} \rightarrow 0 \; (n \rightarrow \infty), $$
and this completes the proof. \\ 

Let $X$ be a complex Banach space and $T \in \mathcal{L}(X)$. Let $r(T)$ be the spectral radius of $T$ and $r_{x}(T)$ be its local spectral radius defined by $r_{x}(T)=\overline {\lim} \|T^nx\|^{1/n}$. We obtain:

\begin{corollary} If $X$ is a complex Banach space and $T \in \mathcal{L}(X)$, then the set of $x$ such that $r_x(T)=r(T)$ has a complement which is $\sigma$-porous, and Haar-null (if $X$ is separable).
\end{corollary}

\begin{proof} Apply the above result to $a_n=1/n$ and use the spectral radius formula: $\lim \|T^n\|^{1/n}=r(T)$.
\end{proof}

\section{Compact case}

In this section, we move away from Haar-negligibility and $\sigma$-porosity and try to improve the condition $a_n \rightarrow 0$. Note that contrarly to Theorem \ref{smallness}, the proof of the next proposition really uses the powers of the operator $T$.

\begin{proposition}\label{compact} Let $X$ be a real or complex Banach space and $T \in \mathcal{L}(X)$. We make two assumptions : \\
i) $T$ is compact. \\
ii) $(\|T^n\|)$ is non-decreasing. \\
Then, for each $\epsilon>0$, there exists $x \in X$, $\|x\| \leqslant 1$ such that for infinitely many $n$'s, we have
$$\|T^nx\| \geqslant (1-\epsilon)\|T^n\|.$$
Furthermore,
$$ \{x \in X, \overline{\lim} \; \dfrac{\|T^nx\|}{\|T^n\|}>0\} $$
is a dense subset of $X$.
\end{proposition}

\begin{proof} There exists $(x_n) \subset X$ with $\|x_n\|=1$ such that $\|T^nx_n\| \geqslant (1-\frac{\epsilon}{2}) \|T^n\|$. By the compactness of $T$, one can extract from $(Tx_n)$ a norm convergent sequence $(Tx_{n_k})$. So, we can find $N$ such that for $k \geqslant N$: 
$$ \|T(x_{n_k})-T(x_N)\| \leqslant \dfrac{\epsilon}{2}. $$
Put $x=x_N$, for $k \geqslant N$, we get
\begin{eqnarray*} \|T^{n_k}x\| & \geqslant & \|T^{n_k}x_{n_k}\| -\|T^{n_k}(x-x_{n_k})\| \\
& \geqslant & (1-\dfrac{\epsilon}{2}) \|T^{n_k}\|-\|T(x-x_{n_k})\| \; \|T^{n_{k}-1}\| \\
& \geqslant & (1-\dfrac{\epsilon}{2}-\|T(x-x_{n_k})\|) \; \|T^{n_k}\| \; \;  {\rm{because}} \; \; (\|T^n\|) \;  \rm{is \; non \; decreasing} \\
& \geqslant & (1-\epsilon)\|T^{n_k}\|.
\end{eqnarray*}
Let us see the density. Take $\eta>0$ and $x \in X$. From what we did, we can find a point $x_0$ with $\|x_0\| \leqslant 1$ such that for infinitely many $n$'s 
$$ \|T^nx_0\| \geqslant \dfrac{1}{2} \|T^n\|. $$
For those $n$, we have
$$ \|T^n(x+\eta x_0)\|+\|T^n(x-\eta x_0)\| \geqslant 2 \eta \|T^nx_0\| \geqslant \eta \|T^n\|. $$
Hence 
$$ \overline{\lim} \; \dfrac{\|T^n(x+\eta x_0)\|}{\|T^n\|}>0 \; \; {\rm{or}} \; \; \overline{\lim} \; \dfrac{ \|T^n(x-\eta x_0)\|}{\|T^n\|}>0 $$
and since $\|x-(x \pm \eta x_0)\| \leqslant \eta$, we get the density.
\end{proof}

The two following examples show that we cannot remove assumptions i) or ii).

\begin{example} i') There exists $T \in \mathcal{L}(\ell^p(\mathbb{N}))$ such that $(\|T^n\|)$ is non-decreasing and for all $x \in \ell^p(\mathbb{N})$, 
$$ \dfrac{\|T^nx\|}{\|T^n\|} \rightarrow 0. $$
ii') There exists a compact operator $T \in \mathcal{L}(\ell^p(\mathbb{N}))$ such that for every $x \in \ell^p(\mathbb{N})$,
$$  \dfrac{\|T^nx\|}{\|T^n\|} \rightarrow 0. $$
\end{example}

\begin{proof} For i'), it is enough to take $T=B$ where $B$ is the unweighted backward shift (see introduction). For ii'), we consider $T$ the weighted backward shift defined on $\ell^p(\mathbb{N})$ with its natural norm and canonical basis $(e_k)$ by 
$$ Tx=\sum_{k=2}^{\infty} w_k x_k e_{k-1}$$
where $(w_k)$ is a sequence decreasing to zero, it is easy to see that $T$ is compact as a norm limit of finite rank operators. For $n \in \mathbb{N}$, we have
$$ T^nx=\sum_{k=n+1}^{\infty} x_k w_k w_{k-1} \cdots w_{k-n+1} e_{k-n}, $$
this implies 
$$ \|T^nx\|^p=\sum_{k=n+1}^{\infty} \left|x_k w_k w_{k-1} \cdots w_{k-n+1}\right|^p \leqslant (\prod_{k=2}^{n+1} w_k)^p \sum_{k=n+1}^{\infty} |x_k|^p . $$
Considering $\|T^ne_{n+1}\|$, we obtain exactly $\|T^n\|^p=\left(\prod_{k=2}^{n+1} w_k \right)^p$ and hence 
$$ \dfrac{\|T^nx\|^p}{\|T^n\|^p} \leqslant \sum_{k=n+1}^{\infty} |x_k|^p \rightarrow 0. $$ 
\end{proof}

If the space $X$ is reflexive, we can slightly improve the previous result.

\begin{proposition} Let $X$ be a real or complex Banach space and $T \in \mathcal{L}(X)$. We make the assumptions i) and ii) of Proposition \ref{compact} and we suppose moreover that $X$ is reflexive. Then, there exists $x \in X$, $\|x\| \leqslant 1$ such that
$$ \overline{\lim} \; \frac{\|T^nx\|}{\|T^n\|}=1.$$
\end{proposition}

\begin{proof} A compact operator always attains its norm on a reflexive space (although this fact is not stricly necessary here). Write $\|T^n\|=\|T^nx_n\|$ with $(x_n) \subset B_X$. From $(x_n)$, we can by reflexivity extract a sequence $(x_{n_k})$ which converges weakly to some point $x$ with $\|x\| \leqslant 1$. Estimating $\|T^{n_k}x\|$ as above, we get
$$ \|T^{n_k}x\| \geqslant (1-\|T(x-x_{n_k})\|)\|T^{n_k}\|. $$
Then we use the (well-known) fact that a compact operator transforms weakly convergent sequences into norm convergent sequences. Hence $\|T(x-x_{n_k})\|$ goes to 0 and 
$$ \overline{\lim} \; \frac{\|T^nx\|}{\|T^n\|} \geqslant 1,$$
which concludes the proof.
\end{proof}

Our last example shows that the reflexivity cannot be omitted.

\begin{example} There exists a compact operator $T$ on $c_0$ (the space of the sequences converging to zero with its usual norm), such that $(\|T^n\|)$ is non decreasing and for each $x \in c_0$ with $\|x\| \leqslant 1$,
$$  \overline{\lim} \; \dfrac{\|T^nx\|}{\|T^n\|}<1 . $$
\end{example}

\begin{proof} We consider this time the operator $T$ defined on $c_0$ by
$$ Tx=\sum_{k=1}^{\infty} w_k x_k e_{k-1}+x_0 e_0 $$
where $(w_n)$ is a sequence decreasing to zero with $w_0=1$. Hence, we see that $T$ is compact (same argument as ii')). From the formulas,
$$ \begin{cases}
Te_0=e_0\\
Te_{k}=w_k e_{k-1} \; \; (k \geqslant 1)
\end{cases} \\ $$
We deduce by induction that: 
$$ T^nx= \left (\sum_{k=0}^{n} W_k x_k \right )e_0+\sum_{k=1}^{\infty} w_{k+n} w_{k+n-1} \cdots w_{k+1} x_{k+n} e_k $$
where we put $ W_n=\prod_{i=0}^{n} w_i. $ Since $(w_n)$ decreases to 0, $(W_n)$ itself decreases to zero faster than any geometric sequence, so $W=\sum_{k=0}^{\infty} W_k<\infty$. On the other hand, we have for $\|x\| \leqslant 1$
\begin{eqnarray*} \|T^nx\| & = & \max \left (\left |\sum_{k=0}^{n} W_k x_k \right |, \sup_{k \geqslant 1} w_{k+n} w_{k+n-1} \cdots w_{k+1} |x_{k+n}|   \right ) \\
& \leqslant & \sum_{k=0}^n W_k.
\end{eqnarray*}
because $0 \leqslant w_i \leqslant 1$. Considering the vector $e_0+\cdots+e_n$, we get 
$ \|T^n\|=\sum_{k=0}^{n} W_k$,
so $(\|T^n\|)$ is indeed non-decreasing. Suppose now that there exists a point $x \in c_0$, $\|x\| \leqslant 1$ such that
$$ \overline{\lim} \; \dfrac{\|T^nx\|}{\|T^n\|}=1. $$
We see that there exists a non-decreasing map $\varphi : \mathbb{N} \rightarrow \mathbb{N}$ such that
$$ \left |\sum_{k=0}^{\varphi(n)} W_k x_k \right| \rightarrow W \; \; (n \rightarrow \infty). $$
Let $N$ be an integer such that $k \geqslant N$, $|x_k| \leqslant 1/2$. Then
$$ \left |\sum_{k=0}^{\varphi(n)} W_k x_k \right| \leqslant \sum_{k=0}^{\infty} |x_k| W_k \leqslant \sum_{k=0}^{N} W_k+\dfrac{1}{2} \sum_{k=N+1}^{\infty} W_k. $$
At the limit when $n \rightarrow \infty$, the above inequality yields
$$ W \leqslant \sum_{k=0}^{N} W_k+\dfrac{1}{2} \sum_{k=N+1}^{\infty} W_k <W,$$ a contradiction.
\end{proof}

\begin{remark} In the last example and in i'), one can also require that $\|T^n\| \rightarrow \infty$ (replace $T$ by $2T$).
\end{remark}

Finally, the last proposition of this section can be seen as a variation of the compact case (see the remark after the proof). It will also be useful to us in section 4.

\begin{proposition} \label{firstcase} Let $X$ be a real or complex Banach space and $T \in \mathcal{L}(X)$ be a non nilpotent operator. Assume there exists $a \in ]0,1[$ and a subspace $M$ of finite codimension such that $\|T^n_{\vert_{M}}\| \leqslant a \|T^n\|$ for infinitely many $n$'s, then
$$\{x \in X, \overline{\lim} \; \frac{\|T^nx\|}{\|T^n\|}>0\} $$
is a dense subset of $X$.
\end{proposition}

\begin{proof} By a previous argument, it is enough to find only one point to have automatic density. Replacing $M$ by its closure, we may assume that $M$ is closed in $X$. Write $X=F \oplus M$ where $F$ is a subspace of finite dimension. Let $(f_1,\cdots,f_r)$ be a normed basis of $F$ and $(f^*_1,\cdots,f^*_r)$ its dual basis (in $F$). For $1 \leqslant i \leqslant r$, extend each $f^*_i$ to $X$ requiring that the restriction of $f^*_i$ to $M$ is 0. If we denote by $P$ the continuous projection onto $F$ (with respect to the previous decomposition), we easily see that $f^*_i$ is continuous with $\|f^*_i\| \leqslant \|{f^*_i}_{\vert_F}\| \|P\|$. Set $C=\sup_{i} \|f^*_i\|$ and choose $a_r>0$ such that $Cra_r+a<1$. Now, set 
$$A=\{n \in \mathbb{N}, \|T^n_{\vert_{M}}\| \leqslant a \|T^n\|\}$$ and
$$A_i=\{n \in A, \|T^nf_i\| \geqslant a_r \|T^n\|\}. $$
We show that $A=\cup_{i=1}^r A_i$ and since  $A$ is infinite by hypothesis, so will be one of the $A_i$ and this will give the conclusion. Suppose on the contrary that there exists $n \in A \setminus \cup_{i=1}^r A_i$. Fix $\alpha_r \in ]Cra_r+a,1[$ and let $x \in X$, $\|x\|=1$ such that $\|T^nx\| \geqslant \alpha_r \|T^n\|$. Write $x=\sum_{i=1}^r x_i f_i+u$ where $u \in M$. By construction of the $f^*_i$, $x_i=f^*_i(x)$, hence $|x_i| \leqslant C$ and we get 
\begin{eqnarray*}  \|T^nx\| & \leqslant & \sum_{i=1}^{r} |x_i| \|T^nf_i\|+\|T^nu\| \\
& \leqslant & (C r a_r +a)\|T^n\|<\alpha_r \|T^n\|,
\end{eqnarray*}
a contradiction.
\end{proof}

\begin{remark} For $T \in \mathcal{L}(X)$, define 
$$ \|T\|_{\mu}=\inf\{\|T_{\vert_{M}}\|, M \subset X, {\rm{codim}} \; M < \infty\}. $$
This quantity measures the degree of non-compactness of $T$ since $\|T\|_{\mu}=0$ if and only if $T$ is compact (see \cite{LS} for details). The above result roughly says that if $\|T^n\|_{\mu}$ is not too large "uniformly" (it is the same $M$ that works for infinitely many $n$'s), then we have the same conclusion as in the compact case.
\end{remark}

\section{Modulus of asymptotic uniform smoothness and optimal exponents}

Our goal in this section is to compute the value of 
\begin{eqnarray*} q_X=\sup \{q>0, {\rm{for \; every \; non \; nilpotent \; and \; bounded \; linear \; operator}} \; T;
\end{eqnarray*}
$$ \exists x \in X, \sum_{n=1}^{\infty}\left( \dfrac{\|T^nx\|}{\|T^n\|} \right )^{q}=\infty\} $$
for some classical Banach spaces. From the introduction we know that that $q_X$ is well-defined (i.e. the set over which we take the supremum is not empty) and $q_X \geqslant 1$. Furthermore, the known values \cite{Mu} are $q_{\ell^1}=1$ and $q_{H}=2$ if $H$ is an Hilbert space. We will compute here $q_{c_0}$, $q_{\ell^p}$ and $q_{L^p(0,1)}$ for $1 \leqslant p<\infty$. Observe that we actually have  $q_X={q'_X}$ where
\begin{eqnarray*} q'_X=\sup \{q>0, {\rm{for \; every \; non \; nilpotent \; and \; bounded \; linear \; operator \;}}  T;
\end{eqnarray*}
$$  {\rm{the \; set}} \; \{x \in X, \sum_{n=1}^{\infty}\left( \dfrac{\|T^nx\|}{\|T^n\|} \right )^q=\infty\} \; {\rm{is \; dense \; in \;}} X\}.$$

Indeed, this follows from the following observation which is a simple consequence of the Baire category theorem.

\begin{proposition}\label{baire} Let $T \in \mathcal{L}(X)$ be non nilpotent, $q_0>0$ and assume that for every $q<q_0$, there exists $x_0 \in X$ such that
$$ \sum_{n=1}^{\infty} \left ( \dfrac{\|T^nx_0\|}{\|T^n\|} \right )^q=\infty.$$
Then the set
$$A=\{x \in X, \; \forall q<q_0, \; \sum_{n=1}^{\infty} \left ( \dfrac{\|T^nx\|}{\|T^n\|} \right )^q=\infty \}$$
is a dense $G_{\delta}$ subset of $X$.
\end{proposition}

\begin{proof} Considering a non-decreasing sequence $(s_k)$ such that $s_k<q_0$ for each $k$ and $s_k \rightarrow q_0$ when $k \rightarrow \infty$, we see that it is enough to show that for each $q<q_0$, $\widetilde{A}$ is a dense $G_{\delta}$ subset of $X$ where
$$ \widetilde{A}=\{x \in X, \sum_{n=1}^{\infty} \left ( \dfrac{\|T^nx\|}{\|T^n\|} \right )^q=\infty \}.$$
Write $\widetilde{A} =\bigcap_{N=1}^{\infty} \Omega_N, $ where 
$$ \Omega_N=\{x \in X, \sum_{n=1}^{\infty} \left ( \dfrac{\|T^nx\|}{\|T^n\|} \right )^q>N\}.$$
Using for example the Fatou-lemma, it is easy to see that $X \setminus \Omega_N$ is a closed set of $X$ for each $N$. By the Baire category theorem, we need to check that each $\Omega_N$ is dense in $X$. Fix $\epsilon>0$ and $x \in X$. Replacing $x_0$ by $\lambda x_0$ for some $\lambda>0$, we can assume that $\|x_0\|=\epsilon$.  Now, there exists a constant $C>0$ such that $(x+y)^q \leqslant C(x^{q}+y^{q})$ for every $x,y \geqslant 0$.  From this and the triangle inequality, we get
$$ \sum_{n=1}^{\infty} \left ( \dfrac{\|T^n(x-x_0)\|}{\|T^n\|} \right )^q+\sum_{n=1}^{\infty} \left ( \dfrac{\|T^n(x+x_0)\|}{\|T^n\|} \right )^q \geqslant \dfrac{2^{q}}{C} \sum_{n=1}^{\infty} \left ( \dfrac{\|T^nx_0\|}{\|T^n\|} \right )^q=\infty. $$
So $x-x_0$ or $x+x_0$ belongs to $\widetilde{A}$ and since $\|x-(x \pm x_0)\|=\epsilon$, this concludes the proof.
\end{proof}

\begin{remark} It is possible (at least in the case where $q>1$ and $X$ is reflexive) to show that under the assumptions of Proposition \ref{baire}, the set $A$ is Haar-null. Indeed, it is enough to show that for each $N$, $X \setminus \Omega_N$ is Haar-null and this follows directly from the theorem of Matou\v{s}kov\`a \cite{Ma}: a closed and convex set of empty interior in a reflexive Banach space is Haar-null.
\end{remark}

To compute $q_X$, we introduce the modulus of asymptotic uniform smoothness of $X$ which is a useful tool from Banach space geometry (but probably not much used in linear dynamics). This quantity has been introduced for the first time by Milman in \cite{Mi} under some different names. We follow here the more recent terminology which can be found in \cite{JLPS}.

\begin{definition} Let $X$ be a real or complex Banach space. The modulus of asymptotic uniform smoothness of $X$ is the fonction $\overline{\rho}_{X}(t)$ defined by 
$$ \overline{\rho}_{X}(t)=\sup_{\|x\|=1} \; \inf_{\dim(X/Y)<\infty} \; \sup_{y \in Y, \|y\|=1} \left( \|x+ty\|-1 \right) \; (t \geqslant 0). $$
\end{definition}
$\overline{\rho}_{X}$ is a $1$-lipschitz, convex and non-decreasing map such that $\overline{\rho}_{X}(0)=0$ and $\overline{\rho}_{X}(t) \leqslant t$ for $t \geqslant 0$. We can now state: 

\begin{theorem} \label{thm:main} Let $X$ be a Banach space and assume that its modulus of asymptotic uniform smoothness satisfies the following property: $\overline{\rho}_{X}(2t)=O(\overline{\rho}_{X}(t))$ when $t \rightarrow 0$. Let $\rho: \mathbb{R}^+ \rightarrow \mathbb{R}^+$ be a non-decreasing map such that $\rho(t)>0$ whenever $t>0$ and 
$$ \lim_{t \rightarrow 0} \dfrac{\overline{\rho}_{X}(t)}{\rho(t)}=0, $$
then there exists a point $x \in X$ such that 
$$  \sum_{n=1}^{\infty} \rho \left( \dfrac{\|T^nx\|}{\|T^n\|} \right )=\infty.$$ 
\end{theorem}

As a consequence, we obtain the results claimed in the Introduction. More precisely:

\begin{theorem} Let $T \in \mathcal{L}(X)$ ($X=\ell^p(\mathbb{N})$, $X=L^p(0,1) (1 \leqslant p<\infty)$ or $X=c_0(\mathbb{N})$) be a non nilpotent operator. Then \\
a) If $X=\ell^p$, the set
$$\{x \in X, \forall q<p, \left (\dfrac{\|T^nx\|}{\|T^n\|} \right )_{n} \notin \ell^q\}$$
is a dense $G_{\delta}$ of $X$. On the other hand, there exists $S \in \mathcal{L}(X)$ (non nilpotent) such that for each $x$, 
$$ \sum_{n=1}^{\infty} \left ( \dfrac{\|S^nx\|}{\|S^n\|} \right )^p<\infty. $$
b) If $X=L^p$, the set 
$$\{x \in X, \forall q<\min(p,2), \left (\dfrac{\|T^nx\|}{\|T^n\|} \right )_{n} \notin \ell^q\}$$ is a dense $G_{\delta}$ of $X$. On the other hand, there exists $R \in \mathcal{L}(X)$ (non nilpotent) such that for each $x$,
$$ \sum_{n=1}^{\infty} \left ( \dfrac{\|R^nx\|}{\|R^n\|} \right )^{\min(p,2)}<\infty. $$
c) If $X=c_0$, the set 
$$\{x \in X, \forall q>0, \left (\dfrac{\|T^nx\|}{\|T^n\|} \right )_{n} \notin \ell^q\}$$ is a dense $G_{\delta}$ of $X$. In particular, the following holds
$$ q_{\ell^p}=p, q_{L^p}=\min(p,2) (1 \leqslant p<\infty), q_{c_0}=\infty.$$
\end{theorem}

\begin{proof} It is known (and easy to see) that $\overline{\rho}_{c_0}(t)=0$ for $0 \leqslant t \leqslant 1$ and for $t \geqslant 0$,
$$ \overline{\rho}_{\ell^p}(t)=(1+t^p)^{1/p}-1 \sim \dfrac{t^p}{p} \; \; (t \rightarrow 0).$$
For $L^p=L^p(0,1)$, Milman \cite{Mi} obtained the following estimates. For $L^1$, $\overline{\rho}_{L^1}(t)=t$. For $1<p<2$
$$ \dfrac{1}{p} t^p \leqslant  \overline{\rho}_{L^p}(t) \leqslant \dfrac{2}{p} t^p \; \; (t \rightarrow 0). $$
For $2<p<\infty$, there exists a constant $C_p>0$ such that
$$ (p-1)t^2 \leqslant \overline{\rho}_{L^p}(t) \leqslant C_p t^2 \; \; (t \rightarrow 0). $$ 
For $p=2$, $\overline{\rho}_{L^2}(t)=(1+t^2)^{1/2}-1$ since $L^2$ and $\ell^2$ are isometric. Hence the first statement of a) and b), and statement c) are a straightforward consequence of Theorem \ref{thm:main} and Proposition \ref{baire}. We now turn to the examples. a) can be found in \cite{Mu}. The job there is done for $p=2$, but the general case is almost the same. We include the example anyways for the sake of completness. Let $(e_i)$ be the usual canonic basis of $\ell^p$ and set $e_{1,0}=e_1$, $e_{1,1}=e_2$, $e_{2,0}=e_3$, $e_{2,1}=e_4$, $e_{2,2}=e_5$ $\ldots$ In this way, we can write $\ell^p=\oplus_{k=1}^{\infty} X_k$ where $X_k$ is the $(k+1)$-dimensional $\ell^p$ space with the basis $e_{k,0},\ldots,e_{k,k}$. Then $S$ is defined by $S=\oplus_{k=1}^{\infty} \frac{1}{2^k}B_k$ where $B_k$ is the usual backward shift on $\mathcal{L}(X_k)$, i.e. $B_k(e_{k,j})=e_{k,j-1}$ for $j \geqslant 1$ and $B_ke_{k,0}=0$. For $n \geqslant 1$, $S^n(e_{n,n})=2^{-n^2} e_{n,0}$ so $\|S
 ^n\| \geqslant 2^{-n^2}$. Let $x_k=\sum_{j=0}^{k} \alpha_j e_{k,j} \in X_k$, we have 
$$ \sum_{n=1}^{\infty} \left ( \dfrac{\|S^nx_k\|}{\|S^n\|} \right )^p \leqslant \sum_{n=1}^{k} \left ( \dfrac{2^{n^2}}{2^{nk}} (\sum_{j=n}^{k} |\alpha_j|^p)^{1/p} \right )^p \leqslant \sum_{n=1}^{k} \dfrac{1}{2^{n(k-n)}} \|x_k\|^p \leqslant 2 \|x_k\|^p. $$
It follows that for $x=\sum_{k=1}^{\infty} x_k$, we have
$$ \sum_{n=1}^{\infty} \left ( \dfrac{\|S^nx\|}{\|S^n\|} \right )^p \leqslant \sum_{k=1}^{\infty} \sum_{n=1}^{\infty} \left ( \dfrac{\|S^nx_k\|}{\|S^n\|} \right )^p \leqslant \sum_{k=1}^{\infty} 2\|x_k\|^p<\infty.$$
Now for b), recall that every $1 \leqslant p<\infty$, $\ell^p$ is isomorphic to a complemented subspace $E \subset L^p$. So write $L^p=E \oplus F$ where $F$ is a closed subspace of $L^p$. If $Q: \ell^p \rightarrow E$ is an isomorphism, then clearly $S_0=QSQ^{-1}$ is a bounded operator on $E$ satisfying for every $x \in E$,
$$ \sum_{n=1}^{\infty} \left ( \dfrac{\|S^n_0x\|}{\|S^n_0\|} \right )^p<\infty. $$
Let $P$ be the projection onto $E$ with respect to the decomposition  $L^p=E \oplus F$ and put $R=S_0 P$ which is bounded on $L^p$. Then for every $n$, $R^n=S^n_0 P$. Since $P$ is a projection onto $E$, we have $\|R^n\| \geqslant \| S^n_0 \|$. Hence 
$$ \sum_{n=1}^{\infty} \left ( \dfrac{\|R^nx\|}{\|R^n\|} \right )^{p} \leqslant \sum_{n=1}^{\infty} \left ( \dfrac{\|S^n_0(Px)\|}{\|S^n_0\|} \right )^{p} <\infty. $$ 
This shows b) for $p \leqslant 2$. For $p \geqslant 2$, use a similar argument and the fact that $\ell^2$ is isomorphic to a complemented subspace of $L^p$.
\end{proof}

It remains to prove Theorem \ref{thm:main}. Before going into the proof, we will need the following elementary lemma.

\begin{lemma}\label{lem:serie} Let $f,g$ be two maps from $\mathbb{R}^+$ to $\mathbb{R}^+$ such that for each $x>0$, $g(x)>0$. Assume that 
$$ \lim_{x \rightarrow 0} \dfrac{f(x)}{g(x)}=0 \; \; {\rm{and}} \; \; \lim_{x \rightarrow 0} f(x)=0. $$
Then, there exists a sequence $(\alpha_i) \subset \mathbb{R}^+$, $\alpha_i \rightarrow 0$ (when $i \rightarrow \infty$) such that 
$$ \sum_{i=1}^{\infty} f(\alpha_i)<\infty \; \; {\rm{and}} \; \; \sum_{i=1}^{\infty} g(\alpha_i)=\infty. $$
\end{lemma}

\begin{proof} If $g(x) \nrightarrow 0$, there exists $\epsilon>0$ and $(\alpha_i) \subset \mathbb{R}^+$, $\alpha_i \rightarrow 0$ such that for each $i \geqslant 1$, $g(\alpha_i) \geqslant \epsilon$. By passing to a subsequence, one can assume that $\sum_{i=1}^{\infty} f(\alpha_i)<\infty$ (because $f(x) \rightarrow 0$) and since $g(\alpha_i) \geqslant \epsilon$ for each $i$, one also gets $\sum_{i=1}^{\infty} g(\alpha_i)=\infty$. Now, if $g(x) \rightarrow 0$, fix $(\epsilon_i) \subset \mathbb{R}^+$ with $\epsilon_i>0$ for each $i$ such that $\sum_{i=1}^{\infty} \epsilon_i<\infty$. Choose for each $i \geqslant 1$, $x_i \in ]0,\frac{1}{i}[$ such that $f(x_i) \leqslant \epsilon_i g(x_i)$ and $g(x_i) \leqslant \frac{1}{2}$. Put $n_1=1$ and $n_{i+1}=n_i+[\frac{1}{g(x_i)}]$ where $[x]$ denotes the integer part of a number $x$. If $k$ is an integer with $n_i \leqslant k \leqslant  n_{i+1}-1$, put $\alpha_k=x_i$, thus $\alpha_k \rightarrow 0$ because $x_i \rightarrow 0$. We have 
\begin{eqnarray*} 
\sum_{j=1}^{\infty} f(\alpha_j)&=&\sum_{i=1}^{\infty} \sum_{k=n_i}^{n_{i+1}-1} f(\alpha_k)= \sum_{i=1}^{\infty} f(x_i)(n_{i+1}-n_i) \\
& \leqslant & \sum_{i=1}^{\infty} \epsilon_i g(x_i)(n_{i+1}-n_i) \leqslant \sum_{i=1}^{\infty} \epsilon_i<\infty.
\end{eqnarray*}
On the other hand, 
\begin{eqnarray*}
\sum_{j=1}^{\infty} g(\alpha_j)&=&\sum_{i=1}^{\infty} g(x_i)(n_{i+1}-n_{i}) \\
& \geqslant & \sum_{i=1}^{\infty}(1-g(x_i)) \geqslant \sum_{i=1}^{\infty} \dfrac{1}{2}=\infty,
\end{eqnarray*} and this concludes the proof of the lemma.
\end{proof}

We are now ready for the proof of Theorem \ref{thm:main}. We will combine some techniques from \cite{Mu} and \cite{L}.

\begin{proof}[Proof of Theorem \ref{thm:main}] We can make the following assumption (*): for every subspace $M$ of finite codimension, $\|T^n_{\vert_{M}}\|>\frac{1}{2} \|T^n\|$ for all but a finite number of $n$'s. Indeed, if this is not true, then by Proposition \ref{firstcase}, there exists $\epsilon>0$, a point $x \in X$  and a non decreasing sequence $(n_j)$ such that for each $j$:
$$  \dfrac{\|T^{n_j}x\|}{\|T^{n_j}\|} \geqslant \epsilon. $$
Since $\rho(\epsilon)>0$, this obviously implies that 
$$ \sum_{n=1}^{\infty} \rho \left( \dfrac{\|T^nx\|}{\|T^n\|} \right )=\infty.$$
Hence we can suppose that (*) holds. By Lemma \ref{lem:serie}, there exists a sequence $(\tilde{\alpha_i})$,  $\tilde{\alpha_i} \rightarrow 0$ such that 
$$ \sum_{i=1}^{\infty} \overline{\rho}_{X}(\tilde{\alpha_i})<\infty \; \; {\rm{and}} \; \; \sum_{i=1}^{\infty} \rho(\tilde{\alpha_i})=\infty. $$
Since $\rho$ is non-decreasing and $\overline{\rho}_{X}(2t)=O(\overline{\rho}_{X}(t))$, we see that 
$$ \sum_{i=1}^{\infty} \overline{\rho}_{X}(\alpha_i)<\infty \; \; {\rm{and}} \; \; \sum_{i=1}^{\infty} \rho(\dfrac{\alpha_i}{2})=\infty $$ where we have put $\alpha_i=2 \tilde{\alpha_i}$. Now, using again the hypothesis $\overline{\rho}_{X}(2t)=O(\overline{\rho}_{X}(t))$, we see that for each $k$, $\sum_{i=1}^{\infty} \overline{\rho}_{X}(2^k \alpha_i)<\infty$ and thus $\prod_{i=1}^{\infty} (1+\overline{\rho}_{X}(2^k \alpha_i))$ converges. From this and the fact that $\alpha_i \rightarrow 0$, we can find an increasing sequence of integers $(m_k)$ such that 
$$ \alpha_i \leqslant 2^{-k} \; (i \geqslant m_k) \; \; {\rm{and}} \; \; \prod_{i=m_k}^{\infty} (1+\overline{\rho}_{X}(2^k \alpha_i)) \leqslant 2. $$
Without loss of generality, we may assume that $m_1=1$. Let us also fix once for all a sequence $(\beta_i) \subset \mathbb{R}^+$, $\beta_i>0$ such that $\prod_{i=1}^{\infty} (1+\beta_i)$ converges. Next, we are going to construct by induction 2 other sequences $(n_j)$ and $(u_i)$  such that: $n_1<n_2<\cdots$, $\|u_i\|=1$. Further, these sequences will have to satisfy two properties. First, for each $l \geqslant 1$ and $j \leqslant l$:

\begin{equation} \label{equation1} \tag{1} 
\left \|T^{n_j} \left ( \sum_{i=1}^{l} \alpha_i u_i \right ) \right \| \geqslant \dfrac{\alpha_j}{2} \|T^{n_j}\|.
\end{equation}

Secondly, for each $k \geqslant 1$, and $m_k \leqslant l \leqslant m_{k+1}-1$:
\begin{equation} \label{equation2} \tag{2} \left \| \sum_{i=m_k}^{l}  \alpha_i u_i \right \| \leqslant 2^{1-k} \left ( \prod_{i=m_k}^{l} (1+\beta_i) \right ) \left ( \prod_{i=m_k}^{l}(1+\overline{\rho}_{X}(2^k \alpha_i)) \right ). 
\end{equation}

Once this is done, by (\ref{equation2}) and since $\prod_{i=m_k}^{\infty} (1+\overline{\rho}_{X}(2^k \alpha_i)) \leqslant 2$, we see that $\left ( \sum_{i=1}^{l} \alpha_i  u_i \right )_l$ is Cauchy, and thus converges to some point $x$. By (\ref{equation1}), for fixed $j$, we get at the limit
$$ \|T^{n_j}x\| \geqslant \dfrac{\alpha_j}{2} \|T^{n_j}\|,$$ whence
$$ \sum_{n=1}^{\infty} \rho \left( \dfrac{\|T^nx\|}{\|T^n\|} \right ) \geqslant \sum_{j=1}^{\infty} \rho \left( \dfrac{\|T^{n_j}x\|}{\|T^{n_j}\|} \right ) \geqslant \sum_{j=1}^{\infty} \rho(\dfrac{\alpha_j}{2})=\infty$$ and this is the desired conclusion. Now, we have to do the induction part of the proof. Set $n_1=1$. There exists $u_1 \in X$ such that $\|u_1\|=1$ and $\|Tu_1\| \geqslant \frac{\alpha_1}{2}$. Let $k \geqslant 1$ and assume the construction has been carried out until a point $m_k \leqslant l \leqslant m_{k+1}-1$. If $l+1=m_{k+1}$, then
$$ \|\alpha_{l+1} u_{l+1} \| \leqslant 2^{-(k+1)} \leqslant 2^{1-(k+1)} (1+\beta_{l+1})(1+\overline{\rho}_{X}(2^{k+1} \alpha_ {l+1})),$$ 
so $(\ref{equation2})$ is automatically satisfied for $l+1$. Arguments to get (\ref{equation1}) will be detailed later. Suppose now that $l+1 \leqslant m_{k+1}-1$. For visibility until the end of the proof, we put $s_l=\sum_{i=m_k}^{l}  \alpha_i u_i$ and $x_l=\sum_{i=1}^{l}  \alpha_i u_i$. We distinguish 2 cases. Suppose first that $\|s_l\| \leqslant 2^{-k}$. Then for any choice of $u_{l+1}$ such that $\|u_{l+1}\|=1$, 
$$ \left \|\sum_{i=m_k}^{l+1} \alpha_i u_i \right \| \leqslant \|s_l\|+\alpha_{l+1} \leqslant 2^{1-k} .$$  Hence (\ref{equation2}) is again always satisfied. We indicate now how to get (\ref{equation1}). For each $j \leqslant l$, select a linear functional $f_j$ such that $\|f_j\|=1$ and $\|f_j(T^{n_j}x_l)\|=\|T^{n_j}x_l\|$ and put 
$$ M=\bigcap_{j=1}^{l} \ker (f_j T^{n_j})$$ which is clearly of finite codimension. Applying (*), we can find $n_{l+1}>n_{l}$ and $v \in X$, $\|v\|=1$ such that  
$$ \|T^{n_{l+1}}v\| \geqslant \dfrac{1}{2} \|T^{n_{l+1}}\|. $$
Since we have
\begin{eqnarray*}
\|T^{n_{l+1}}(x_l+\alpha_{l+1} v)\|+\|T^{n_{l+1}}(x_l-\alpha_{l+1} v)\|& \geqslant & 2 \|T^{n_{l+1}}v\| \alpha_{l+1} \\
& \geqslant & \|T^{n_{l+1}}\| \alpha_{l+1},
\end{eqnarray*}
there exists $\epsilon=\pm 1$ such that 
$$ \|T^{n_{l+1}}(x_l+\epsilon \alpha_{l+1} v)\| \geqslant \dfrac{\alpha_{l+1}}{2} \|T^{n_{l+1}}\|.$$
Putting $u_{l+1}=\epsilon v$, we have proved (\ref{equation1}) for $j=l+1$. For $j \leqslant l$, we obtain
\begin{eqnarray*} \| T^{n_j}(x_l+\epsilon_{l+1} \alpha_{l+1} u_{l+1})\| & \geqslant & \|f_jT^{n_j}(x_l+\epsilon_{l+1} \alpha_{l+1} u_{l+1}) \| \\
& = & \|f_jT^{n_j}x_l\|=\|T^{n_j}x_l\| \\
& \geqslant & \dfrac{\alpha_j}{2} \|T^{n_j}\|,
\end{eqnarray*}
where the last inequality follows from the induction hypothesis of (\ref{equation1}). Suppose now that $\|s_l\| \geqslant 2^{-k}$. From the definition of the modulus of asymptotic smoothness, we can find a subspace $Y \subset X$ of finite codimension such that for all $y \in Y$, $\|y\|=1$
\begin{equation} \label{equation3} \tag{3}
\left \| \dfrac{s_l}{\|s_l\|}+2^k \alpha_{l+1}y \right \| \leqslant (1+\beta_{l+1})(1+\overline{\rho}_{X}(2^k \alpha_{l+1})).
\end{equation}
This time, we set
$$ M=\bigcap_{j=1}^{l} \ker (f_j T^{n_j}) \; \bigcap \; Y$$ where the $f_j$ are constructed exactly as in the previous case. We also construct in the same way $n_{l+1}$ and $u_{l+1} \in M$, $\|u_{l+1}\|=1$ so that (\ref{equation1}) holds for $l+1$. Now (\ref{equation3}) with $y=u_{l+1} \in M \subset Y$ yields

\begin{equation} \label{equation4} \tag{4}
 \left \|s_l+2^k\|s_l\|\alpha_{l+1} u_{l+1} \right \| \leqslant \|s_l\| (1+\beta_{l+1})(1+\overline{\rho}_{X}(2^k \alpha_{l+1})). 
\end{equation} 
The condition  $\|s_l\| \geqslant 2^{-k}$ implies that $s_l+\alpha_{l+1} u_{l+1}$ is on the segment joining $s_l$ with $s_l+2^k \|s_l\|\alpha_{l+1} u_{l+1}$. Hence, we get from (\ref{equation4})

\begin{eqnarray*}
\left \|\sum_{i=m_k}^{l+1} \alpha_i u_i \right \|& = & \|s_l+\alpha_{l+1} u_{l+1}\| \\
 & \leqslant & \|s_l\|(1+\beta_{l+1})(1+\overline{\rho}_{X}(2^k \alpha_{l+1})) \\
 & \leqslant & 2^{1-k} \left ( \prod_{i=m_k}^{l+1} (1+\beta_i) \right ) \left ( \prod_{i=m_k}^{l+1}(1+\overline{\rho}_{X}(2^k \alpha_i)) \right ),
\end{eqnarray*}
where the last inequality follows from the induction hypothesis. We see that (\ref{equation2}) is satisfied for $l+1$ and this ends the construction.
\end{proof}

\begin{remark} What Lindenstrauss \cite{L} showed is the following result: if $(x_i) \subset X$ is such that $\sum_{i=1}^{\infty}\epsilon_i x_i$ diverges for every choice of signs $(\epsilon_i) \subset \{-1,1\}$, then $\sum_{i=1}^{\infty} \rho_X(\|x_i\|)=\infty$ where $\rho_X$ is the usual modulus of smoothness defined for $t \geqslant 0$ by
$$ \rho_X(t)=\dfrac{1}{2} \sup_{\|x\|=1, \|y\|=1} \; (\|x+ty\|+\|x-ty\|-2). $$
\end{remark}

\begin{remark} We do not know if the assumption $\overline{\rho}_{X}(2t)=O(\overline{\rho}_{X}(t))$ is really necessary although it seems to us that a "bad" Orlicz space may not satisfy this property.
\end{remark}


\begin{thebibliography}{HD}
\bibitem[Bal1]{Bal1} K.M. Ball, The plank problem for symmetric bodies, Invent. Math. 10 (1991),
535--543.
\bibitem[Bal2]{Bal2} K.M. Ball, The complex plank problem, Bull. London Math. Soc. 33 (2001),
433--442.
\bibitem[Bay1]{Bay1} F. Bayart, Porosity and hypercyclic operators, Proc. Amer. Math. Soc. 133 (2005), 3309--3316. 
\bibitem[Bay2]{Bay2} F. Bayart, Weak-closure and polarization constant by Gaussian meausure, Math Zeitschrift 264 (2010), 459--468.
\bibitem[Be]{Be} B. Beauzamy, Introduction to operator theory and invariant subspaces, North-
Holland Mathematical Library Vol. 42, North-Holland, Amsterdam (1988).
\bibitem[BL]{BL} Y. Benyamini and J. Lindenstrauss : Geometric Nonlinear Functional Analysis, AMS
Colloquium Publications 48 (2000).
\bibitem[BM]{BM} F. Bayart and E. Matheron, Dynamics of linear operators, Cambridge Tracts in Mathematics (2009).
\bibitem[BMM]{BMM} F. Bayart and E. Matheron and P. Moreau, Small sets and hypercyclic vectors, Comment. Math. Univ. Carolinae 49 (2008), 53--65.
\bibitem[C]{C} J. P. R. Christensen, On sets of Haar measure zero in Polish abelian groups. Israel J. Math. 13 (1972), 255--260.
\bibitem[D]{D} E.P. Dol\v{z}enko, Boundary properties of arbitrary functions, Izv. Akad. Nauk SSSR Ser. Mat. 31 (1967), 3--14.
\bibitem[HS]{HS} B. R. Hunt, T. Sauer and J. A. Yorke, Prevalence: a translation invariant "almost every" on infinite-dimensional spaces. Bull. Amer. Math. Soc. 27 , no. 2 (1992), 217--238.
\bibitem[JLPS]{JLPS} W. B. Johnson, J. Lindenstrauss, D. Preiss and G. Schechtman, Almost Fr\'echet differentiability of Lipschitz mappings between infinite-dimensional Banach spaces, Proc.
London Math. Soc. (3) 84 (2002), 711--746.
\bibitem[L]{L} J. Lindenstrauss, On the modulus of smoothness and divergent series in Banach spaces, Mich. Math. J. 10 (1963), 241--252. 
\bibitem[LS]{LS} A. Lebow and M. Schechter, Semigroups of operators and measures of noncompactness, J. Funct. Anal. 7 (1971), 1--26.
\bibitem[Ma]{Ma} E. Matou\v{s}kov\`a, Translating finite sets into convex sets, Bull. Lond. Math. Soc. 33, No. 6 (2001), 711--714.
\bibitem[Mi]{Mi} V. D. Milman, Geometric theory of Banach spaces. II. Geometry of the unit ball,
Uspekhi mat. Nauk 26 , no. 6 (162) (1971), 73--149 (Russian), Russian Math. Surveys
26 no. 6 (1971), 79--163 (English).
\bibitem[Mu]{Mu} V. M\"uller, Orbits, weak orbits and local capacity of operators, Integral
Equations Operator Theory 41 (2001), 230--253.
\bibitem[MV]{MV} V. M\"uller and J. Vr{\v{s}}ovsk\'y, On orbit-reflexive operators, J.London Math.Soc. 79 (2009), 497--510.
\bibitem[R]{R} C. Read, The invariant subspace problem for a class of Banach spaces. II. Hypercyclic operators, Israel J.Math. 63 (1988), 1--40.
\bibitem[Z]{Z} L. Zaj\i\v{c}ek, Porosity and $\sigma$-porosity, Real Anal. Exchange 13 no. 2 (1987-1988), 314--350.
\end{thebibliography}
\end{document}